\documentclass[12pt]{article}

\usepackage[leqno]{amsmath}
\usepackage{amsthm,amscd}
\usepackage{amsfonts} % to get the \mathbb alphabet
\usepackage{amssymb} % for subsetneq

\newtheorem{theorem}{Theorem}[section]
\newtheorem{prop}[theorem]{Proposition}
\newtheorem{co}[theorem]{Corollary}
\newtheorem{lm}[theorem]{Lemma}
\newtheorem{rem}[theorem]{Remark}

\newtheorem{definition}[theorem]{Definition}

\def\Proof{\noindent{\sl Proof.}\qquad}

\def\rank{\mathop{\rm rank}}

\newcommand{\trans}{^T}

\def\kmms{\kern-\mathsurround}

\newcommand{\calB}{{\cal B}}

\newcommand{\calA}{{\cal A}}

\newcommand{\bbR}{{\mathbb R}}

\newcommand{\bbZ}{{\mathbb Z}}

\newcommand{\period}{{\rho}}
\renewcommand{\mod}{\mathop{\rm mod}}
\newcommand{\im}{\mathop {\rm im}}

\begin{document}

\title{Periodicity of non-central integral arrangements 
modulo positive integers}
\author{
Hidehiko Kamiya
\footnote
{
{\rm This work was partially supported by 
%the Ministry of Education, Culture, Sports, Science and Technology, 
%MEXT 
JSPS 
KAKENHI (19540131).}\,\, 
%Grant-in-Aid for Scientific Research (C), 19540131, 2007.}\,\, 
{\it Faculty of Economics, Okayama University}
} 
\\ 
Akimichi Takemura 
\footnote
{
{\it Graduate School of Information Science and Technology,
University of Tokyo} 
} 
\\
Hiroaki Terao 
\footnote
{
%{\rm 
%This work was supported by the MEXT and the JSPS.}\,\, 
%{\tt hterao00@za3.so-net.ne.jp}\,\,\,
{\it Department of Mathematics, Hokkaido University}
}
}
%\author{Hidehiko Kamiya
%\\
%{\it Faculty of Economics, Okayama University}\\
%Akimichi Takemura \\
%{\it Graduate School of Information Science and Technology}\\
%{\it University of Tokyo}\\
%and \\
%Hiroaki Terao \\
%{\it Department of Mathematics, Hokkaido University}
%}
%\date{\today}
\date{March 2008}

\maketitle

\begin{abstract}
An integral coefficient matrix determines an integral 
arrangement of hyperplanes in $\bbR^m$. 
%for some $m\in \bbZ_{>0}$.  
After modulo $q$ reduction ($q\in \bbZ_{>0}$), the same matrix 
determines an arrangement $\calA_q$ 
of ``hyperplanes'' in $\bbZ_q^m$. 
In the special case of central arrangements,  
Kamiya, Takemura and Terao %\cite{ktt} 
[{\it J. Algebraic Combin.}, to appear]  
showed that 
the cardinality of the complement of $\calA_q$  
in $\bbZ_q^m$ is a quasi-polynomial in $q\in \bbZ_{>0}$. 
Moreover, they proved in the central case that the intersection lattice of 
$\calA_q$ is periodic from some $q$ on. 
The present paper generalizes these results to the case of 
non-central arrangements. 
The paper also studies the arrangement $\hat{\calB}_m^{[0,a]}$ 
of Athanasiadis %\cite{ath99} 
[{\it J. Algebraic Combin.} {\bf 10} (1999), 207--225]  
to illustrate our results. 

\end{abstract}

%\maketitle

\section{Introduction} 

An $m\times n$ integral coefficient matrix 
$C\in {\rm Mat}_{m\times n}(\bbZ)$ 
and a vector $b\in \bbZ^n$ of integral constant terms 
determine an arrangement $\calA$ 
of $n$ hyperplanes 
$H_j, \ 1\le j\le n$, in $\bbR^m$. 
In the same way, 
for a positive integer $q$, the modulo $q$ reductions 
of $C$ and $b$ determine an arrangement $\calA_q=\calA_q(C, b)$ 
of $n$ ``hyperplanes'' 
$H_{j, q}, \ 1\le j\le n$,  
in $\bbZ_q^m$, where $\bbZ_q:=\bbZ/q\bbZ$.   
In the special case $b=0$, which we call the central case, 
we showed in \cite{ktt} that  
the cardinality of the complement 
$M(\calA_q)$ of 
this arrangement in $\bbZ_q^m$ is a 
quasi-polynomial in $q\in \bbZ_{>0}$. 
In \cite{ktt} we called this quasi-polynomial the 
%%HT 20080318 \em
{\em
characteristic 
quasi-polynomial},  
and gave an explicit %expression of a 
period $\period_0$ (called the lcm period in \cite{ktt-root}) of 
this %characteristic 
quasi-polynomial, 
although a ``period collapse'' might occur and 
$\period_0$ may not be the minimum period in general. 
The characteristic quasi-polynomials of the arrangements of 
root systems and the mid-hyperplane arrangements (\cite{kott}) 
are studied in \cite{ktt-root}.  
In the present paper, we consider the general case where 
$b$ is not necessarily zero, which we call the 
%%HT 20080318 \em
{\em
non-central case}.   

In the non-central case, we show that the cardinality of the 
complement 
of the arrangement after modulo $q$ reduction 
%in $\bbZ_q^m$ 
remains to be a quasi-polynomial 
in $q$ with $\period_0$ as a period. 
However, unlike the central case $b=0$, this periodicity holds 
true not for all $q>0$ %\in \bbZ_{>0}$ 
but for all $q>q_0$ for some $q_0\in \bbZ_{\ge 0}$. %in general. 
We give a bound $q_0$ explicitly, though it may not be the 
strict one %bound 
in general. 

Besides, for each $q\in \bbZ_{>0}$, 
we can define a poset $L_q=L_q(C, b)$, 
called the 
%%HT 20080318 \em
{\em
intersection poset}, 
consisting of nonempty 
intersections of some of $H_{j,q}, \ 1\le j\le n$, with partial 
order defined by reverse inclusion.  
Then we can consider a sequence $L_1, L_2,\ldots$ and study 
its periodicity. 
With an appropriate definition of an isomorphism of 
the intersection posets, %$L_q, \ q\in \bbZ_{>0}$, 
the sequence of isomorphism classes of 
$L_q, \ q=1,2,\ldots$, %\in \bbZ_{>0}$, 
is shown to be periodic from some $q$ on 
(Section \ref{sec:posets}).  
%It has to be recognized that 
%We have to be aware of 
We have to recognize 
the distinction between the periodicity 
of %the characteristic quasi-polynomial 
$|M(\calA_q)|$ and that of 
%the intersection posets 
$L_q$. 
This distinction is illustrated with a simple example 
in Section \ref{sec:posets}. %\ref{sec:Bm}. 
%in the same section. 
  
For the concepts in the theory of arrangements of hyperplanes, 
the reader is referred to \cite{ort}. 
%For general properties of 
Concerning general properties of 
quasi-polynomials, 
%the reader should consult \cite{sta}.  
\cite{sta} is a basic reference.  

%poset 

\

The organization of this paper is as follows. 
In Section \ref{sec:elementary-divisor} we 
%%HT 20080318
%%develop 
present
a basic theory of
elementary divisors modulo $q$.
In Section \ref{sec:quasi-poly}, we prove that the cardinality of the complement 
of the arrangement after modulo $q$ reduction 
%in $\bbZ_q^m$ 
is a quasi-polynomial for all sufficiently large $q$. 
In Section \ref{sec:posets}, we 
%show 
investigate 
the periodicity of the intersection posets.  
In the final section, Section \ref{sec:Bm}, we 
study the arrangement $\hat{\calB}_m^{[0,a]}$ of 
Athanasiadis \cite{ath99} 
%a deformation of the Coxeter arrangement of type $B_m$ 
to illustrate our general results in the non-central 
case. 
%illustrate our results with the{\bf (an)} example of {\bf ???}. 

\section{Elementary divisors modulo a positive integer}
\label{sec:elementary-divisor}

In this section we 
%%HT 20080318 
%develop 
present
a basic theory of canonical forms and
elementary divisors of matrices with entries in $\bbZ_q$.  This is
needed in our developments in Sections \ref{sec:quasi-poly} and
\ref{sec:posets}.

% Let $m, n, q \in \bbZ_{>0}$ be positive integers. 
Suppose an $m\times n$ %integral 
matrix 
$
A
\in {\rm Mat}_{m\times n}(\bbZ_{q} )
$ 
is given.   The dimensions $m,n$ in 
this section are general and are not necessarily
equal to those in other sections.
Let
\[
GL_{k} (R) :=
\{
M \in 
{\rm Mat}_{k\times k}(R): 
\det M {\rm ~is~a~unit~in~}
R 
\} 
%GL_{k} (\bbZ_{q}) :=
%\{
%M \in 
%{\rm Mat}_{k\times k}(\bbZ_{q} )
%\mid
%\det M {\rm ~is~a~unit~in~}
%\bbZ_{q} 
%\}. 
\]
for an arbitrary commutative ring $R$. 
We say that $A$ is equivalent to $B \in {\rm Mat}_{m\times n}(\bbZ_{q} )$,
denoted by
$A\sim B$,
if $B = PAQ$ for some
$P\in GL_{m} (\bbZ_{q})$   
and
$Q\in GL_{n} (\bbZ_{q})$.   
Our purpose is to 
find a canonical form of $A$
using elementary divisors.
For $a %b
\in \bbZ$, we denote its $q$ reduction by 
$[a %b
]_q:=a%b
+q\bbZ\in \bbZ_q$.
For an integral matrix or vector $A'$, 
let $[A']_q$ stand for the element-wise $q$ reduction of 
$A'$.

\begin{prop}
\label{prop:1}
Let $A\in {\rm Mat}_{m\times n}(\bbZ_{q})$ be 
an $m\times n$ matrix with entries in $\bbZ_q$. 
Then,  
\begin{enumerate}
\item[(i)] 
%(1)
The matrix $A$ is equivalent to
\[
{\rm diag} ([d_{1}]_{q}, [d_{2}]_{q}, \dots, [d_{s}]_{q}, 
0, 0 , \dots , 0)
%{\rm diag} ([e_{1}]_{q}  , [e_{2}]_{q}  , \dots , [e_{\ell}]_{q} , 
%0, 0 , \dots , 0)
\]
for some integers $d_1,d_2,\ldots,d_s$ 
such that 
$0 < d_{1} \leq d_{2} \leq \dots \leq d_{s} < q$ 
and
$d_{1} \mid
d_{2} \mid
\dots \mid
d_{s} \mid
q.
$
\item[(ii)] 
%(2) 
The integers 
$d_{1}, d_{2}, \dots, d_{s}$ 
%$e_{1},
%e_{2}, 
%\dots,
%e_{\ell}
%$ 
are uniquely determined by $A$. 
\end{enumerate}
\end{prop}

We call $d_1, \ldots, d_s\in \bbZ$, 
or $[d_1]_q,\ldots, [d_s]_q \in \bbZ_q$, 
the 
%%HT 20080318 \em
{\em
elementary divisors}
 of 
$A\in {\rm Mat}_{m\times n}(\bbZ_q)$. 

\begin{proof} 
Let $A' \in {\rm Mat}_{m\times n}(\bbZ)$ 
with $A = [A']_{q}. $ 
By the theory of elementary divisors over $\bbZ$, 
there exist
 $P'\in GL_{m}(\bbZ) $ 
and
 $Q'\in GL_{n}(\bbZ) $ 
 such that 
$P'A'Q' =
{\rm diag} (e_{1}, e_{2}, \dots , e_{\ell}, 
0, 0 , \dots , 0)
%{\rm diag} (b_{1}, b_{2}, \dots , b_{s}, 
%0, 0 , \dots , 0)
$
with 
$e_{1} \mid e_{2} \mid \dots \mid e_{\ell}, \ 
1\leq e_{1} \leq e_{2} \leq \dots \leq e_{\ell}$. 
%$b_{1} \mid b_{2} \mid\ldots\mid b_{s} $, 
%$1\leq b_{1} \leq b_{2} \leq\ldots\leq b_{s}$.
Note that 
 $$
{\rm gcd} \{a, q \}
=
{\rm gcd} \{b, q \}
\Longleftrightarrow
[a]_{q} 
\doteq
[b]_{q} 
\Longleftrightarrow
[a]_{q} \bbZ_{q} 
=
[b]_{q} \bbZ_{q}
$$
for $a, b \in\bbZ_{>0}. $ 
Here $\doteq$ stands for the equality up to a unit multiplication
%%HT 20080318
in $\bbZ_{q}$.
Define
\[
s
%\ell
=
\max
\{
j: 
q
{\rm ~does~ not~ divide~}
e_j
%b_{j} 
\},
\qquad 
%\,\,\,\,\,
d_i=
%e_{i} =
{\rm gcd} \{ e_i, %b_{i}, 
q \} 
\ \ (1\le i\le s). 
%\,\,\,\,
%(i = 1, 2, \dots , s). %\ell).
\]
Then 
$
[d_i]_q
%[e_{i} ]_{q} 
\doteq
[e_i]_q, \  
%[b_{i} ]_{q} 
%\,\,\,\,
%(i = 1, 2, \dots , s)%r)
1\le i\le s
$,
which proves (i). %(1). 
The uniqueness (ii) %(2) 
follows from Lemma \ref{2} below. 
\end{proof} 

\begin{lm}
\label{2}
%%HT 20080318  A' -> B, \alpha' -> \beta in this Lemma
Let $R$  be an arbitrary commutative ring.  
Let
$A$ and  $B$ be  
both $m\times n$ matrices with entries in $R$.  
Suppose that
$$
A
=
{\rm diag}(\alpha_{1}, \alpha_{2}, \dots, \alpha_{s}, 0, \dots , 0)
\,\,\,
{\rm and}
\,\,\,
B
=
{\rm diag}(\beta_{1}, \beta_{2}, \dots, \beta_{t}, 0, \dots , 0)
$$
satisfy
$$
\alpha_{1} R 
\supseteq 
\alpha_{2} R 
\supseteq 
\dots
\supseteq 
\alpha_{s} R 
\neq 
(0)
,
\,\,\,\,\,\,\,
\beta_{1} R 
\supseteq 
\beta_{2} R 
\supseteq 
\dots
\supseteq 
\beta_{t} R 
\neq 
(0)
$$ 
and
$
B
=
P A Q^{-1} 
$ 
with
$
P\in GL_{m} (R)
, \ 
Q\in GL_{n} (R).
$ 
Then 
$$s = t \,\,\,\,\,\,{\rm and}
\,\,\,\,\,\,
\alpha_{i} R
=
\beta_{i} R 
\ \ (1\le i\le s). 
%\quad
%(i = 1, 2, \dots , s).
$$
\end{lm}

\begin{proof} 
We may assume $m\leq n$ without loss of generality.
Define 
$$
\alpha_{s+1} 
=
\dots
=
\alpha_{m} 
=
\beta_{t+1} 
=
\dots
=
\beta_{m} 
=
0.
$$ 
Let $1\leq k \leq m$.
Define 
$I_{k} $ to be the ideal of $R$ generated by
$
\{
p_{ij}: %\mid
1\leq j\leq k \leq i \leq m
\},
$  
where $p_{ij} $ is the 
$(i, j)$-entry
of $P$.  
(For example,
$I_{1} $ is the ideal generated by
all entries in the first column of $P$.)
Then
$\det P \in I_{k} $ for each $k$
because the product
$
p_{1\sigma(1)} 
p_{2\sigma(2)}
\dots 
p_{m\sigma(m)} 
$ 
for every permutation $\sigma$ 
belongs to the ideal $I_{k} $.
Since
$\alpha_{j} p_{ij} $ 
is the $(i, j)$-entry of
$PA = B Q$,
one has
$\alpha_{j} p_{ij} 
\in
\beta_{i} R.
$ 
Thus
\[
\alpha_{k} p_{ij} R \subseteq
\alpha_{j} p_{ij} R \subseteq
\beta_{i} R \subseteq
\beta_{k} R 
\]
for
$1\leq j\leq k\leq i\leq m$.
This shows
$
\alpha_{k} (\det P) \in \beta_{k} R
$
and
$
\alpha_{k} R \subseteq \beta_{k} R.
$ 
The converse is symmetric.
\end{proof}

For  $A \in {\rm Mat}_{m\times n}(\bbZ_{q} )$, let 
$\im A$ denote the  submodule of the $\bbZ_q$-module $\bbZ_{q}^m$ generated by the
columns of $A$.  For $b\in \bbZ_{q}^m$, we are interested in
determining whether $b\in \im A$ or not.  Let 
$(A,b) \in {\rm Mat}_{m\times (n+1)}(\bbZ_{q} )$ be an $m\times (n+1)$ matrix 
%in which $b$ is appended 
obtained by appending $b$ 
to $A$ as the last column.

\begin{lm}
\label{lm:subgroup}
For $A \in {\rm Mat}_{m\times n}(\bbZ_{q} )$ and 
$b\in \bbZ_{q}^m$, we have 
$b\in \im A$ if and only if the elementary divisors of $(A,b)$ are the
same  as those of $(A, 0)$.
\end{lm}

\begin{proof}
Suppose that $b\in \im A$.  Then by elementary column operations, 
we can transform $(A,b)$ to $(A,0)$.  Therefore, the elementary divisors
of $(A,b)$  are the same as those of 
$(A, 0)$. 
%$A$.

Suppose that both 
$(A,b)$ and $(A, 0)$ have the same elementary divisors 
$d_1, \ldots, d_s$. 
%$e_{1},\dots,e_{\ell}$.
Then $\im (A, b)$  and
$\im (A, 0)$ have the same cardinality 
%cardianality 
$q^s/(d_1 \cdots d_s)$. 
%$q^\ell/(e_{1}\cdots e_{\ell})$.
Since $\im (A,b) \supseteq \im (A, 0)$ always
holds, we have
$\im (A,b) = \im (A, 0)$.
\end{proof}

\section{Characteristic quasi-polynomial in non-central case} 
%of non-central arrangement modulo $q$}
\label{sec:quasi-poly}

Let $m, n \in \bbZ_{>0}$ be positive integers. 
Suppose an $(m+1)\times n$ integral matrix 
$$
%A=
\begin{pmatrix}
C \\ 
b
\end{pmatrix}
\in {\rm Mat}_{(m+1)\times n}(\bbZ)
$$ 
is given, where 
$C=(c_1,\ldots,c_n)\in {\rm Mat}_{m\times n}(\bbZ)$ 
with 
$c_j=(c_{1j},\ldots,c_{mj})\trans\ne (0,\ldots,0)\trans, \ 1\le j\le n$,  
and $b=(b_1,\ldots,b_n)\in \bbZ^n$.  
% For $q\in \bbZ_{>0}$, write 
% $[b_j]_q:=b_j+q\bbZ\in \bbZ_q, \ %=\bbZ/q\bbZ, \ 
% [c_j]_q:=([c_{1j}]_q,\ldots,[c_{mj}]_q)\trans$, 
Define 
$$
H_{j,q}=H_{q}(c_j, b_j):=\{ z=(z_1,\ldots,z_m)\in \bbZ_q^m: z[c_j]_q= [b_j]_q \} 
$$
and 
$$
\calA_q=\calA_{q}(C, b):=\{ H_{j,q}: 1\le j\le n\}. 
$$ 
When $q$ is not a prime, it is not appropriate to call 
$H_{j,q}$ a hyperplane, but by abuse of terminology we 
call $H_{j,q}$ a 
%%HT 20080318
{\em
hyperplane} also in such cases.  
Denote the complement of $\calA_q$ by 
$$ 
M(\calA_q)
:=\bbZ_q^m\setminus \bigcup_{H\in \calA_q}H  
=\{ z\in \bbZ_q^m: z[C]_q-[b]_q\in (\bbZ_q^{\times})^n\}, 
$$ 
where 
% $[C]_q:=([c_1]_q,\ldots,[c_n]_q), \ 
% [b]_q:=([b_1]_q,\ldots,[b_n]_q)$ and 
$\bbZ_q^{\times}:=\bbZ_q\setminus \{ 0 %[0]_q 
\}$.  
Then we have 
\begin{equation}
\label{eq:|M|=q+}
|M(\calA_q)|
=q^m+\sum_{\emptyset \ne J\subseteq [n]}(-1)^{|J|}|H_{J,q}| 
\end{equation}
with  
$$
H_{J,q}%=H_{J,q}(b_J)
:=\bigcap_{j\in J}H_{j,q} 
\quad 
\text{ for \ } 
J\subseteq [n]:=\{ 1,2,\ldots,n\}  
$$
(In Section \ref{sec:posets}, we have to consider 
$H_{J,q}$ for $J=\emptyset$, in which case 
we understand that $H_{\emptyset,q}=\bbZ_q^m$).   

For $J=\{ j_1,\ldots,j_{|J|}\}\subseteq [n], \ 
1\le |J|\le n, \ j_1<\cdots<j_{|J|}$, 
denote  
\begin{eqnarray*}
C_J&:=&(c_{j_1},\ldots,c_{j_{|J|}})\in {\rm Mat}_{m\times |J|}(\bbZ), \\ 
b_J&:=&(b_{j_1},\ldots,b_{j_{|J|}})\in \bbZ^{|J|}. 
\end{eqnarray*}  
Using $C_J$ and $b_J$, we can write $H_{J,q}$ %=H_{J,q}(b_J)$ 
as 
$$
H_{J,q}=
H_{q}(C_J, b_J):=\{ z\in \bbZ_q^m: z[C_J]_q=[b_J]_q\}.  
$$ 
%where $[C_J]_q:=???, \ [b_J]_q:=([b_{j_1}]_q,\ldots,[b_{j_{|J|}}]_q)$. 
Now, let $f_J: \bbZ^m \to \bbZ^{|J|}$ be a $\bbZ$-homomorphism 
defined by $z \mapsto zC_J$, and 
\begin{equation}
\label{eq:fJq}
f_{J,q}: \bbZ_q^m \to \bbZ_q^{|J|}
\end{equation} 
the induced morphism 
$z \mapsto z[C_J]_q$. 
When $[b_J]_q\in {\rm im}f_{J,q}$, i.e., 
$z_0[C_J]_q=[b_J]_q$ for some $z_0\in \bbZ_q^m$,  
we have 
$H_q(C_J, b_J)=z_0+H_q(C_J,0)
:=\{ z_0+z: z\in H_q(C_J,0) \}$;  
when $[b_J]_q\notin {\rm im}f_{J,q}$, on the other hand, 
we have $H_q(C_J, b_J)=\emptyset$.  
Hence  
\begin{equation} 
\label{eq:H=H-b-in-im}
%(|H_{J,q}|=)
|H_{q}(C_J, b_J)|=
\begin{cases}
|H_{q}(C_J, 0)| & \text{if $[b_J]_q\in {\rm im}f_{J,q}$,} \\ 
0 & \text{otherwise.} 
\end{cases}
\end{equation}

Let 
$$
A_J:=
\begin{pmatrix}
C_J \\ 
b_J 
\end{pmatrix}
\in {\rm Mat}_{(m+1)\times |J|}(\bbZ).  
$$
By Lemma \ref{lm:subgroup}, we know that 
$[b_J]_q\in {\rm im}f_{J,q}$ if and only if
the elementary divisors of $[C_J]_q$ are the same as those of
$[A_J]_q$.   As seen in the proof of Proposition \ref{prop:1}, 
the elementary divisors of $[A_J]_q$ and $[C_J]_q$ are obtained 
by $q$ reduction of the 
%ordinary elementary divisors of $A_J$ and $C_J$, respectively. 
elementary divisors of $A_J$ and $C_J$ over $\bbZ$, respectively. 
Let $%0  <  
e_{J,1} | e_{J,2} | \cdots | e_{J,\ell(J)}$ denote the elementary
divisors of $C_J$  and let 
$%0 < 
e'_{J,1} | e'_{J,2} | \cdots | e'_{J, \ell'(J)}$ denote the elementary divisors of 
$A_J$. 
%$C_J$.  
Denote $d_{J,j}(q):=\gcd\{ e_{J,j}, q\}, \ 1\le j\le \ell(J)$, 
%$j=1, \dots,\ell(J)$,
and $d'_{J,j}(q):=\gcd\{ e'_{J,j}, q\}, \ 1\le j\le \ell'(J)$. 
%$j=1,\dots,\ell'(J)$.  
The elementary divisors of $[C_J]_q$ are 
$d_{J,j}, \ 1\le j\le s(J)$,
%j=1,\ldots,s(J)$, 
where 
\[
s(J)=s(J,q):=
\max \{ j: 1\le j\le \ell(J), \ q\not| \ e_{J,j} \}
=\max\{ j: 1\le j\le \ell(J), \ d_{J,j}(q)<q\}.   
\] 
%\[
%d_{J,j}(q), \ j=1,\dots, \ell(J), \ \text{s.t.}\   d_{J,j}(q) < q.
%\]
Similarly, the elementary divisors of $[A_J]_q$ are 
$
d'_{J,j}(q), \ 1\le j\le s'(J)$, 
%$j=1,\dots, s'(J) %\bar \ell(J)
%$, 
with $s'(J)=s'(J,q):=
\max\{ j: 1\le j\le \ell'(J), \ q \not| \ e'_{J,j}\}$. 
%with $\bar d_{J,j}(q) < q$.
Note %also 
that
$\ell(J)=\rank C_J$, 
$\ell'(J)=\rank A_J$ and
that 
$\ell'(J) - \ell(J) = 0  \text{ or }  1$.

First, consider the case of $J$ for which 
$\rank A_J=\rank C_J$. 
%i.e., $\tilde{b}^2_J=0$.  
%$\tilde{b}_J^2=0$. 
%$\tilde{b}_{J,j}= 0$ for all  $j=\ell(J)+1,\ldots,|J|$. 
For those $J$, we have 
the following equivalence: 
\begin{eqnarray}
\label{eq:equiv-d=d}
&& s(J)=s'(J) \text{ and } 
d_{J,j}(q)=d'_{J,j}(q), \ 1\le j\le s(J) \\ 
&& \qquad \iff 
d_{J,j}(q)=d'_{J,j}(q), \ 1\le j\le \ell(J). 
\nonumber 
\end{eqnarray}
(The implication $\Longrightarrow$ is seen by noting that 
for any $j>s(J)=s'(J)$ we have $d_{J,j}(q)=q=d'_{J,j}(q)$.)
%This equivalence is proved as follows. 
%Suppose $s(J,q)=s'(J,q)$ and 
%$d_{J,j}(q)=d'_{J,j}(q), \ 1\le j\le s(J)$. 
%Then, for any $j>s(J,q)=s'(J,q)$, we have 
%$d_{J,j}(q)=q=d'_{J,j}(q)$. 
%Hence $d_{J,j}(q)=d'_{J,j}(q)$ for all 
%$q=1,\ldots,\ell(J)$.  
%The converse is trivial (Conversely, suppose 
%$d_{J,j}(q)=d'_{J,j}(q), \ 1\le q\le \ell(J)$. 
%Then we have the equivalence 
%$d_{J,j}(q)=q \iff d'_{J,j}(q)=q$ 
%%equivalences 
%%$q|e_{J,j} \iff q%=\gcd\{ e_{J,j}, q\}
%%=d_{J,j}(q)
%%=\bar{d}_{J,j}(q) \iff q|\bar{e}_{J,j}$ for 
%for $j=1,\ldots,\ell(J)$, and thus 
%obtain $s(J,q)=s'(J,q)$).  
Then, \eqref{eq:H=H-b-in-im} and 
\eqref{eq:equiv-d=d} imply 
\begin{equation}
\label{eq:H=IH}
|H_{q}(C_J, b_J)|= 
% 1_{ d_{J,j}(q) |\tilde{b}_{J,j}, \forall j\le \ell(J)} \times |H_{q}(C_J, 0)|, 
\begin{cases}
|H_q(C_J, 0)| & \text{ if } 
%d_{J,j}(q)|\tilde{b}_{J,j} \text{ for all } j=1,\ldots,\ell(J), \\ 
d_{J,j}(q)=d'_{J,j}(q) \text{ for all } j=1,\ldots,\ell(J), \\ 
0 & \text{ otherwise.} 
\end{cases}
\end{equation}
%where $1_P=1$ if $P$ is true and $1_P=0$ otherwise. 
In Lemma 2.1 of \cite{ktt}, we showed that  
\begin{equation}
\label{eq:H=dq}
|H_{q}(C_J, 0)|=d_J(q)q^{m-\ell(J)},
\end{equation}
where $d_J(q):=\prod_{j=1}^{\ell(J)} d_{J,j}(q)$.
% d_{J,j}(q)$. 
% d_{J,j}(q):=\gcd\{ e_{J,j}, q\}
By \eqref{eq:H=IH} and \eqref{eq:H=dq}, we get 
\begin{equation}
\label{eq:H=dq^}
|H_{J,q}|=
|H_{q}(C_J, b_J)|=\tilde{d}_J(q)q^{m-\ell(J)},
\end{equation}
where 
$$
\tilde{d}_J(q)=
\begin{cases}
d_J(q) & \text{ if } 
%d_{J,j}(q)|\tilde{b}_{J,j} \text{ for all } j=1,\ldots,\ell(J), \\ 
d_{J,j}(q)=d'_{J,j}(q) \text{ for all } j=1,\ldots,\ell(J), \\ 
0 & \text{ otherwise.} 
\end{cases}
$$
Now, write $e(J):=e_{J,\ell(J)}$ and define 
$$
\period_0:={\rm lcm}\{ e(J): J\subseteq [n], \ J\ne \emptyset \}. 
$$
Then $d_{J,j}(q+\period_0)=\gcd\{ e_{J,j}, q+\period_0\}
=d_{J,j}(q)$ 
because $e_{J,j} | \period_0, \ 1\le j\le \ell(J)$. 
Moreover, 
$e'(J):=e'_{J,\ell'(J)}|e(J)|\period_0$ 
by Lemma 2.3 of \cite{ktt}, 
and hence 
$d'_{J,j}(q+\period_0)=d'_{J,j}(q), \ 
1\le j\le \ell(J)$, in a similar manner. 
Therefore,  
\begin{equation}
\label{eq:d=d}
\tilde{d}_{J}(q+\period_0)=\tilde{d}_{J}(q)
\end{equation} 
for all nonempty $J\subseteq [n]$. 

Next, consider the case of $J$ for which 
$\rank A_J=\rank C_J+1$. 
%Denote $\tilde e(J)=\bar e_{J,\bar \ell(J)}=
%\bar e_{J,\ell(J)+1}$.
Then for any $q>e'(J)%\tilde{e}(J)
$, there are 
$s'(J)=\ell'(J)=
\ell(J)+1$ elementary divisors of $[A_J]_q$, whereas
there are 
only $s(J)\le \ell(J)$ 
%no more than $\ell(J)$ 
elementary divisors of $[C_J]_q$.
Thus we can conclude 
\begin{equation}
\label{eq:H=0}
|H_{J,q}|=
|H_{q}(C_J, b_J)|=0 \text{ \ for all } 
q>e'(J). %\tilde{e}(J). 
\end{equation}

Equations \eqref{eq:|M|=q+}, \eqref{eq:H=dq^} 
and \eqref{eq:H=0} yield 
\begin{equation}
\label{eq:|M|=sum}
|M(\calA_q)|
=q^m+\sum_{J: \, \rank A_J=\rank C_J}
(-1)^{|J|} \tilde{d}_J(q)q^{m-\ell(J)} 
\end{equation}
for all $q\in \bbZ_{>0}$ with 
\begin{equation}
\label{eq:q>max=q0}
q> 
\max\{ e'(J): %\tilde{e}(J): 
\rank A_J=\rank C_J+1, \ J\ne \emptyset \}
%\max_{J: \, \rank A_J=\rank C_J+1} \tilde{e}(J) 
%\max\{ q_0(J): \tilde{b}_J^2\ne 0 
%\tilde{b}_{J,j}\ne 0 \text{ for some } j=\ell(J)+1,\ldots,|J| 
%\}
=:q_0   
\end{equation} 
%where the summation in \eqref{eq:|M|=sum} 
%is taken over all nonempty $J\subseteq [n]$ such that 
%$\tilde{b}_J^2=0$. 
%
%$\tilde{b}_{J,j}=0$ for all $j=\ell(J)+1,\ldots,|J|$. 
(When $\{ J\subseteq [n]: \rank A_J=\rank C_J+1, \ 
J\ne \emptyset\}=\emptyset$, we understand 
that $q_0=0$). 
Note that whether $\rank A_J=\rank C_J$ or 
$\rank A_J=\rank C_J+1$ 
does not depend on $q$. 
\begin{rem} 
Consider the case of a central arrangement: $b=0$. 
In that case, we have (i) $\{ J: \rank A_J=\rank C_J\}$ 
equals the 
set of all nonempty subsets $J\subseteq [n]$, 
and (ii) $\tilde{d}_J(q)=d_J(q)$ because $d_{J,j}(q)=d'_{J,j}(q)$
%|\tilde{b}_{J,j}=0$ 
for all $j=1,\ldots,\ell(J)$. 
Therefore, in the case of a central arrangement, 
\eqref{eq:|M|=sum} with \eqref{eq:q>max=q0} 
reduces to the result obtained in \cite{ktt}. 
\end{rem}
%Note that whether the condition 
%$\tilde{b}_J^2=0$ 
%%$\tilde{b}_{J,j}=0, \ \forall j>\ell(J)$ 
%holds or not does not depend on $q$, 
%although the same condition, and thus the value $q_0$, do 
%depend on the choice of the unimodular matrices $T_J$.  
%Theoretically, we should take $T_J$'s so that $q_0$ is minimum, 
%and we assume $T_J$'s are taken that way.  
%Practically, however, we might have to be content to use 
%$$
%q_{00}
%:=\max_{J: \, \rank A_J=\rank C_J+1} q_0(J)
%\ge q_0 
%$$ 
%%:=\max_{J\ne \emptyset} q_0(J) 
%%=\max_{1\le |J| \le m+1} q_0(J)$ 
%%:=\max\{ q_0(J): J\ne \emptyset \}
%%=\max\{ q_0(J): 1\le |J| \le m+1 \}$ 
%instead of $q_0$ 
%(Note that by \eqref{eq:SAT}, we have 
%$\rank A_J=\rank C_J+1$ for any $J$ with $\tilde{b}_J^2\ne0$).  
When actually calculating $q_{0}$, we can 
do without %dealing with 
$J$'s with $|J|> m+1$: 
\begin{equation}
\label{eq:q0<=m+1}
q_{0}=
%\max_{J: \, \rank A_J=\rank C_J+1, \ |J|\le m+1} \tilde{e}(J). 
\max\{ e'(J): %\tilde{e}(J): 
\rank A_J=\rank C_J+1, \ 1\le |J|\le m+1 \}. 
\end{equation}
Equation \eqref{eq:q0<=m+1} follows from the following argument: 
For any $J$ with $\rank A_J=\rank C_J+1$ and $|J|>m+1$, 
we can find a subset $J'\subset J$ 
such that $|J'|=\rank A_{J'}=\rank A_J \ (\le m+1)$. 
This $J'$ satisfies $\rank A_{J'}=\rank C_{J'}+1$ because 
$\rank A_{J'}=\rank A_{J}=\rank C_{J}+1\ge \rank C_{J'}+1$.
Now, since $J'\subset J$ and $\rank A_{J'}=\rank A_J$, 
Lemma 2.3 of \cite{ktt} implies that   
$e'(J)%\tilde{e}
|e'(J')%\tilde{e}(J')
\ne 0$. %and thus $q_0(J)\le q_0(J')$.  
Therefore, we have \eqref{eq:q0<=m+1}.   

Now, \eqref{eq:|M|=sum} together with 
\eqref{eq:d=d} implies that 
$|M(\calA_q)|$ is a quasi-polynomial in $q>q_0$ with a period $\period_0$. 
In fact, it is a monic integral quasi-polynomial of degree $m$. 
Furthermore, since $d_{J,j}(\gcd\{\period_0, q\})
=\gcd\{ e_{J,j}, \gcd\{\period_0, q\}\}
=\gcd\{ e_{J,j}, \period_0, q\}
=d_{J,j}(q), \ 1\le j\le \ell(J)$,  
and 
$d'_{J,j}(\gcd\{\period_0, q\})
=d'_{J,j}(q), \ 1\le j\le \ell'(J)$, 
%and hence  
we have 
$$
\tilde{d}_J(q)=\tilde{d}_J(\gcd\{ \period_0, q\})%, 
$$
when $\ell(J)=\ell'(J)$. 
So we can see from \eqref{eq:|M|=sum} that the constituents of the quasi-polynomial 
$|M(\calA_q)|, \ q>q_0$, coincide for all $q$ with the same value 
$\gcd\{ \period_0, q\}$. 

By the discussions so far, we obtain the following theorem: 

\begin{theorem}
\label{th:|M|}
There exist monic polynomials 
$P_1(t),\ldots,P_{\period_0}(t) \in \bbZ[t]$ of degree $m$ 
such that 
$|M(\calA_q)|=P_r(q) \ (q\in r+\period_0\bbZ_{\ge 0}, \ 
1\le r\le \period_0)$ 
for all integers $q>q_0$. 
Moreover, polynomials $P_r(t) \ (1\le r\le \period_0)$ 
depend on $r$ only through $\gcd\{ \period_0, r\}$.   
\end{theorem}

A period $\period_0$ is the same period that was 
used in the central case in \cite{ktt} and \cite{ktt-root}. 
In \cite{ktt-root} this period was called the lcm period. 
Theorem \ref{th:|M|} implies that this $\period_0$ continues 
to be a period of $|M(\calA_q)|, \ q\in \bbZ_{>0}$, for the 
general case of non-centrality.  
However, 
%In contrast, 
unlike the central case $b=0$,  
we have to ignore 
$q\le q_0$ when $b\ne 0$. 
%When $b\ne 0$, 
This exclusion of small $q$'s in the case $b\ne 0$ 
is actually needed in general.  
%though $q_0$ may not be the strict bound. 
This can be seen from the following simple example. 
 
Let $m=1, \ n=2, \ C=(1,1)$ and $b=(1,-1)$. 
Then $H_{1,q}=\{ [1]_q\}$ %\{ z\in \bbZ_q: z=[1]_q\}$ 
and $H_{2,q}=\{ [-1]_q\}$. %\{ z\in \bbZ_q: z=[-1]_q\}$.  
Thus $H_{1,q}=H_{2,q}$ if and only if $q=1,2$.  
Therefore, 
$$
|M(\calA_q)|=
\begin{cases}
q-1 & \text{for $q=1,2$,} \\ 
q-2 & \text{for $q\ge 3$.} 
\end{cases}
$$
This expression implies that $|M(\calA_q)|$ is 
a polynomial (a quasi-polynomial with the minimum period one) 
in $q\in \bbZ_{>0}$ for $q>2$ but not for all 
$q\ge 1$. 
We can calculate 
$q_0=2$, %and $q_{00}=2$, and thus $q_0$ has to be 
from which we see that $q_0$ is the strict bound for $q$ 
in this example. 
In addition, $\period_0=1$, so 
$\period_0$ equals the minimum period in this case. 
%, which gives the strict bound for $q$.  

%In general, however, $q_0$ is not guaranteed to 
%be the strict bound.  

\section{Periodicity of intersection posets}
\label{sec:posets}

In this section, we study the periodicity of the intersection 
posets. 

For each $q\in \bbZ_{>0}$, 
the intersection poset is defined to be the set 
$$
L_q=L_q(C, b):=\{ H_{J,q}\ne \emptyset : J \subseteq [n] \}
$$ 
equipped with the partial order by reverse inclusion. 
When considering the periodicity of the 
sequence $L_1, L_2,\ldots$, 
we have to be careful about the definition of an isomorphism of 
$L_q, \ q\in \bbZ_{>0}$. 
By way of example, let us consider the following simple case. 

Let $m=1, \ n=2$, and 
consider the central case 
%Consider the coefficient matrix 
\begin{equation}
\label{eq:A=3400}
%A=
\begin{pmatrix}
C \\  
b
\end{pmatrix}
=
\begin{pmatrix}
3 & 4 \\  
0 & 0 
\end{pmatrix}. 
\end{equation}
For this pair of coefficient matrix %$C$ 
and constant vector %$b$ 
\eqref{eq:A=3400}, 
%$A$, 
we have 
$\period_0={\rm lcm}\{ 3,4\}=12$,  
and $H_{\{ 1,2\}, q}=\{ 0 \}$ for all $q\in \bbZ_{>0}$  
(Note that $3z=4z=0$ implies $z=4z-3z=0$). 
The intersection posets 
$L_q=L_q(C, b), \ q\in \bbZ_{>0}$, are given as follows:  
\begin{equation}
\label{eq:V<H<H}
\begin{cases}
V<H_{1,q}=H_{2,q}=H_{\{ 1,2\},q} %=\{ 0 \}
& \text{ $\gcd\{ 12, q\}=1$,} \\ 
V<H_{2,q}=\{ 0, \frac{q}{2}\}<H_{1,q}=H_{\{ 1,2\},q} %=\{ 0 \} 
& \text{ $\gcd\{ 12, q\}=2$,} \\ 
V<H_{1,q}=\{ 0, \frac{q}{3}, \frac{2q}{3}\}<H_{2,q}
=H_{\{ 1,2\},q} %=\{ 0 \} 
& \text{ $\gcd\{ 12, q\}=3$,} \\ 
V<H_{2,q}=\{ 0, \frac{q}{4}, \frac{2q}{4}, \frac{3q}{4}\}
<H_{1,q}=H_{\{ 1,2\},q} %=\{ 0 \} 
& \text{ $\gcd\{ 12, q\}=4$,} \\ 
V<H_{1,q}=\{ 0, \frac{q}{3}, \frac{2q}{3}\}<H_{\{ 1,2\},q}, %=\{ 0 \}, 
\ 
V<H_{2,q}=\{ 0, \frac{q}{2}\}<H_{\{ 1,2\},q} %=\{ 0 \} 
& \text{ $\gcd\{ 12, q\}=6$,} \\ 
V<H_{1,q}=\{ 0, \frac{q}{3}, \frac{2q}{3}\}<H_{\{ 1,2\},q}, %=\{ 0 \}, 
\ 
V<H_{2,q}=\{ 0, \frac{q}{4}, \frac{2q}{4}, \frac{3q}{4}\}<H_{\{ 1,2\},q} %=\{ 0 \} 
& \text{ $\gcd\{ 12, q\}=12$,} \\ 
\end{cases}
\end{equation}
where $V:=\bbZ_q^m$. %{\bf (check)} 
Here, we are writing, e.g., $\{ 0, \frac{q}{3}, \frac{2q}{3}\}$ 
instead of  
$\{ [0]_q, [\frac{q}{3}]_q, [\frac{2q}{3}]_q\}$ for simplicity. 
 
According to the usual definition 
of a poset isomorphism,  
$L_q$'s are isomorphic to one another for all $q$'s 
with $\gcd\{ 12, q\}=2, 3$ or $4$, and so are $L_q$'s 
for all $q$'s with $\gcd\{ 12, q\}=6$ or $12$. 
However, here we do not want to consider $L_q$ for 
$\gcd\{ 12, q\}=2, 4:$ 
$$
V<H_{2,q}<H_{1,q}=H_{\{ 1,2\},q}
$$ 
and $L_q$ for $\gcd\{ 12, q\}=3:$ 
$$
V<H_{1,q}<H_{2,q}=H_{\{ 1,2\},q}
$$ 
to be isomorphic to each other. 
This is because of the following. 

We are dealing not with one intersection poset  
but with a sequence of intersection posets $L_1, L_2, \ldots$  
%derived 
obtained for %by(from, under) 
a %common 
fixed numbering of hyperplanes $H_{j,q}, \ j \in [n]$.  
Thus, it is not appropriate to allow a permutation of 
indices $j\in [n]$ tailored for each $L_q, \ q\in \bbZ_{>0}$, 
separately.  

On the other hand, because our concern is the 
periodicity of the sequence $L_1, L_2,\ldots$, 
we may take the numbering of hyperplanes 
as a given one, and do not have to care about 
what the fixed numbering 
on which the sequence $L_1, L_2,\ldots$ is based 
is. 

Based on these considerations, we adopt, 
in this paper, the following definition of 
the isomorphism of the intersection posets $L_q=L_q(C, b), \ 
q\in \bbZ_{>0}$. 

\begin{definition} 
\label{def:iso}
Intersection posets 
$L_q=\{ H_{J,q}\ne \emptyset: J \subseteq [n]\}$ 
and $L_{q'}=\{ H_{J,q'}\ne \emptyset: 
J \subseteq [n]\}, \ q, q'\in \bbZ_{>0}$, 
are defined to be 
isomorphic to each other iff the following conditions 
hold true: 
$$
%H_{J,q}=\emptyset \iff H_{J,q'}=\emptyset  
H_{J,q}\in L_q \iff H_{J,q'}\in L_{q'}   
$$
for all $J \subseteq [n]$, and  
$$
H_{J_1,q}\le H_{J_2,q} \iff H_{J_1,q'}\le H_{J_2,q'} 
$$
for all $J_1, J_2 \subseteq [n]$ such that 
%$H_{I,q}, H_{J,q}, H_{I,q'}, H_{J,q'}\ne \emptyset$.   
$H_{J_1,q}, H_{J_2,q}\in L_q$ and $H_{J_1,q'}, H_{J_2,q'}\in L_{q'}$.   
\end{definition}

For our example \eqref{eq:A=3400}, we see from 
\eqref{eq:V<H<H} that 
$L_q$ with $\gcd\{ 12, q\}=2$ or $4$ is not 
isomorphic to $L_q$ with $\gcd\{ 12, q\}=3$.

We are now in a position to investigate 
the periodicity of the sequence of 
isomorphism classes of $L_q, \ q\ge 1$. 
%$L_q, \ q=1,2,\ldots$. {\bf (4 periods)}
We continue to assume 
$c_j=(c_{1j},\ldots,c_{mj})\trans \ne (0,\ldots,0)\trans$ 
for all $j=1,\ldots,n$. 
Let  
\begin{eqnarray*}
q_1
%&:=&\max\{ e(J): \rank C_J
%=\rank C_{J\setminus \{ j\}}+1, \ 
%j\in J, \ J\setminus \{ j\} \ne \emptyset \} \\ 
%&=&\max\{ e(J): \rank C_J
%=\rank C_{J\setminus \{ j\}}+1, \ 
%j\in J, \ 2\le |J|\le m \}, 
%{\bf (\le m \ check)}
&:=&\max\{ e(J\cup \{ j\}): \rank C_{J\cup \{ j\}}
=\rank C_{J}+1, \ 
j\in [n], \ J \ne \emptyset \} \\ 
&=&\max\{ e(J\cup \{ j\}): \rank C_{J\cup \{ j\}}
=\rank C_{J}+1, \ 
j\in [n], \ 1\le |J|\le m-1 \} 
\end{eqnarray*}
(When there is no pair $(J, j)$ satisfying 
$\rank C_{J\cup \{ j\}}=\rank C_{J}+1$, 
we understand that $q_1=0$). 
%Define $q^*:=\max\{ q_0, q_1, 1\}$.  
Define 
\[
q^*:=\max\left\{ q_0, q_1, \max_{1\le j\le n}\gcd\{ c_{1j},\ldots,c_{mj} \}\right\}\ge 1, 
\]
where $\gcd\{ c_{1j},\ldots,c_{mj} \}, \ 1\le j\le n$, are taken to be positive.  
%
%In the rest of this section, we make the additional assumption 
%that $c_i=(c_{1i},\ldots,c_{mi})\trans \ne (0,\ldots,0)\trans, \ 
%b_i, \ 1\le i\le n$, are taken so that 
%$\gcd\{c_{1i},\ldots,c_{mi},b_i\}=1$ for all $i=1,\ldots,n$.   
Then we can prove the following theorem. 

\begin{theorem} 
\label{th:periodicity-poset}
%The sequence of isomorphism classes of $L_q, \ q=1,2,\ldots$, is 
%periodic in $q>q^*$ with a period $\period_0$, 
Suppose $q, q'\in \bbZ_{>0}$ satisfy $q,q'>q^*$ 
and $\gcd\{\period_0,q\}=\gcd\{\period_0,q'\}$.  
Then we have the following: 
\renewcommand{\labelenumi}{(\roman{enumi})}
\begin{enumerate}
\item 
For any $J\subseteq [n]$, we have  
$H_{J,q}=\emptyset$ if and only if $H_{J,q'}=\emptyset$. 
\item For any $j\in [n]$ and any $J\subseteq [n]$ 
such that $H_{J,q}\ne \emptyset$ and $H_{J,q'}\ne \emptyset$, 
we have 
$H_{j,q}\supseteq H_{J,q}$ if and only if 
$H_{j,q'}\supseteq H_{J,q'}$. 
\end{enumerate}
\end{theorem}

In order to verify this theorem, we first present the following lemma. 

\begin{lm}
\label{lm:period-poset} 
Fix an arbitrary $q\in \bbZ_{>0}$. 
Then, for any $j\in [n]$ and any $J\subseteq [n]$ such that 
$H_{J,q}\ne \emptyset$, 
we have $H_{j,q}\supseteq H_{J,q}$ if and only if 
the following two conditions hold true: 
\renewcommand{\labelenumi}{(\alph{enumi})}
\begin{enumerate}
\item $([b_j]_q, [b_J]_q) 
%(=([b_j]_q, [b_{j_1}]_q,\ldots,[b_{j_{|J|}}]_q) )
\in \bbZ_q^{|J|+1}$ 
lies in the submodule of the $\bbZ_q$-module $\bbZ_q^{|J|+1}$ 
generated by the rows of $([c_j]_q, [C_J]_q) 
\in {\rm Mat}_{m\times (|J|+1)}(\bbZ_q)$; 
%lies in the row space of $([c_j]_q, [C_J]_q) 
%%(=([c_j]_q, [c_{j_1}]_q,\ldots,[c_{j_{|J|}}]_q) )
%\in {\rm Mat}_{m\times (|J|+1)}(\bbZ_q)$ 
%in $\bbZ_q^{|J|+1}$; 
%
\item $[c_j]_q\in \bbZ_q^m$ lies in 
the submodule of the $\bbZ_q$-module $\bbZ_q^m$ 
generated by the columns of 
$[C_J]_q\in {\rm Mat}_{m\times |J|}(\bbZ_q)$.  
%the column space of $[C_J]_q\in {\rm Mat}_{m\times |J|}(\bbZ_q)$  
%in $\bbZ_q^{m}$. 
\end{enumerate}
%
%$[c_j]_q$ lies in the column space of 
%$[C_J]_q\in {\rm Mat}_{m\times |J|}(\bbZ_q)$.  
%%in $\bbZ_q^{m}$
\end{lm}

\Proof 
First suppose $H_{j,q}\supseteq H_{J,q}$. 
%Since 
%$H_{J,q}%=H_q(C_J, b_J)
%=\{ z\in \bbZ_q^m: z[C_J]_q=[b_J]_q \} 
%$H_{j,q}\cap H_{J,q} 
%\ne \emptyset$, 
Then there exists $z_0\in H_{j,q} \cap H_{J,q}
=H_{J,q}\ne \emptyset$. 
Since $z_0([c_j]_q, [C_J]_q)%=(z_0[c_j]_q, z_0[C_J]_q)
=([b_j]_q, [b_J]_q)$, we see that condition (a) holds.  
Further, we have 
%$z\in H_{J,q}$  
%if and only if 
%$z-z_0\in H_q(C_J, 0)$ for $z\in \bbZ_q^m$. 
$H_{j,q}=H_q(c_j,b_j)=z_0+H_q(c_j, 0)$ and 
$H_{J,q}=H_q(C_J, b_J)=z_0+H_q(C_J, 0)$. 
%
%Since $H_{J,q}=\{ z\in \bbZ_q^m: z[C_J]_q=[b_J]_q \} 
%\ne \emptyset$, 
%there exists $z_0\in H_{J,q}$, i.e., 
%$z_0\in \bbZ_q^{m}$ such that 
%$z_0[C_J]_q=[b_J]_q$.   
%Then we have 
%$z[C_J]_q=[b_J]_q$  
%if and only if 
%$(z-z_0)[C_J]_q=0$ for $z\in \bbZ_q^m$. 
%
%Then we have the following equivalences: 
%\begin{eqnarray*}
%z[C_J]_q=[b_J]_q 
%&\iff& 
%(z-z_0)[C_J]_q=0, \quad z\in \bbZ_q^m, \\ 
%z[c_j]_q=[b_j]_q 
%&\iff& 
%(z-z_0)[c_j]_q=0, \quad z\in \bbZ_q^m. 
%\end{eqnarray*} 
Hence $H_{j,q}\supseteq H_{J,q}$ is equivalent to  
$H_q(c_j, 0)\supseteq H_q(C_J, 0)$, which in turn 
is equivalent to 
$[c_j]_q$ being in 
%the column space of $[C_J]_q$ 
the submodule of $\bbZ_q^m$ 
generated by the columns of $[C_J]_q$ 
(Proposition 3.2 of \cite{ktt}). 
Thus condition (b) holds. 
%Then we have $z_0\in H_{j,q}$, so $z_0[c_j]_q=[b_j]_q$.  
%Thus $z[c_j]_q=[b_j]_q$  
%if and only if 
%$(z-z_0)[c_j]_q=0$ for $z\in \bbZ_q^m$. 
%Now, since $H_{j,q}\supseteq H_{J,q}$, we have that 
%$z[C_J]_q=[b_J]_q$ implies $z[c_j]_q=[b_j]_q$, i.e.,  
%$(z-z_0)[C_J]_q=0$ implies $(z-z_0)[c_j]_q=0$, 
%for $z\in \bbZ_q^m$. 

Next, suppose conditions (a) and (b) hold. 
Then by condition (a), we have 
$H_{j,q} \cap H_{J,q} \ne \emptyset$. 
Hence, by the same argument as above, 
$H_{j,q}\supseteq H_{J,q}$ is equivalent to 
%$[c_j]_q$ being in the column space of $[C_J]_q$. 
condition (b). 
So $H_{j,q}\supseteq H_{J,q}$ holds true. 
\qed 

\

\noindent
{\sl Proof of Theorem \ref{th:periodicity-poset}.}\qquad 
%For any $J\subseteq [n]$, 
%we have 
%$H_{J,q}=\{ z\in \bbZ_q^m: z[C_J]_q=[b_J]_q \} 
%\ne \emptyset$ 
%if and only if 
%$[b_J]_q\in \bbZ_q^{|J|}$ 
%lies in 
%the submodule of $\bbZ_q^{|J|}$ generated by the rows of 
%$[C_J]_q$. 
%Here, we can use an argument similar to the proof 
%of Theorem 3.1 of \cite{ktt}, and 
%find that the latter condition is equivalent to 
%$b_JT_J\in \bbZ^{|J|}$ being 
%in the submodule of the $\bbZ$-module $\bbZ^{|J|}$ generated 
%by the rows of 
%$\diag(d_{J,1}(q),\ldots,d_{J,\ell(J)}(q),0,\ldots,0)
%\in {\rm Mat}_{|J|\times |J|}(\bbZ)$ because 
%$q>q_T$. 
%Now, since $q'$ also satisfies $q'>q_T$ and 
%$d_{J,j}(q)=d_{J,j}(q'), \ 1\le j\le \ell(J)$, 
%we have 
%$H_{J,q}=\emptyset$ if and only if $H_{J,q'}=\emptyset$. 
{\sl Part (i):} 
Since $H_{\emptyset,q}=\bbZ_q^m$ and 
$H_{\emptyset,q'}=\bbZ_{q'}^m$, the equivalence 
$H_{J,q}=\emptyset \Leftrightarrow H_{J,q'}=\emptyset$ is 
trivially true for $J=\emptyset$. 
So let us consider nonempty $J\subseteq [n]$. 
We have 
$H_{J,q}=\{ z\in \bbZ_q^m: z[C_J]_q=[b_J]_q \} \ne \emptyset$ 
if and only if $[b_J]_q\in {\rm im}f_{J,q}$, where $f_{J,q}$ is 
defined in \eqref{eq:fJq}. 
First, consider the case of $J$ for which 
$\rank A_J=\rank C_J$. 
Then we know 
%from  \eqref{eq:b-in-imf<->d|b} 
that $[b_J]_q\in {\rm im}f_{J,q}$ 
%if and only if $d_{J,j}(q)|\tilde{b}_{J,j}$ for all $j=1,\ldots,\ell(J)$. 
if and only if 
%$d_{J,i}(q)|\tilde{b}_{J,i}$ for all $i=1,\ldots,\ell(J)$. 
$d_{J,i}(q)=d'_{J,i}(q)$ for all $i=1,\ldots,\ell(J)$. 
% j no juufuku wo sakeru tame 
This argument is valid when $q$ is replaced by $q'$. 
Now, since 
%$d_{J,j}(q)=d_{J,j}(q'), \ 1\le j\le \ell(J)$, 
$d_{J,i}(q)=d_{J,i}(q')$ 
% j no juufuku wo sakeru tame 
and $d'_{J,i}(q)=d'_{J,i}(q') \  
(1\le i\le \ell(J))$ 
because of 
the assumption $\gcd\{ \period_0,q\}=\gcd\{ \period_0,q'\}$, 
we find that $H_{J,q}=\emptyset$ if and only if 
$H_{J,q'}=\emptyset$.  
Next, consider the case of $J$ for which 
$\rank A_J=\rank C_J+1$.  
In that case, we have by \eqref{eq:H=0} that 
$H_{J,q}=H_{J,q'}=\emptyset$ because 
$q,q'>q^*\ge e'(J)%\tilde{e}(J)
$, so 
the equivalence 
$H_{J,q}=\emptyset \Leftrightarrow H_{J,q'}=\emptyset$ is 
trivially true. 

{\sl Part (ii):} 
%Let $j\in [n]$, and suppose $J\subseteq [n]$ 
%satisfies $H_{J,q}\ne \emptyset$ and $H_{J,q'}\ne \emptyset$.  
%We want to verify $H_{j,q}\supseteq H_{J,q} \Leftrightarrow 
%H_{j,q'}\supseteq H_{J,q'}$. 
%
Thanks to $q,q'>q^*\ge \gcd\{ c_{1j},\ldots,c_{mj} \} \ge 1$, 
%Thanks to the assumptions 
%$\gcd\{ c_{1j},\ldots,c_{mj}, b_j\}=1$ and 
%$q,q'>q^*\ge 1$, 
we have $H_{j,q}\ne \bbZ_q^m$ and $H_{j,q'}\ne \bbZ_{q'}^m$, 
so neither $H_{j,q}\supseteq H_{J,q}$ nor 
$H_{j,q'}\supseteq H_{J,q'}$ can happen for $J=\emptyset$; 
therefore, we may assume $J\ne \emptyset$. 
Moreover, when $j\in J$, the equivalence 
$H_{j,q}\supseteq H_{J,q} \Leftrightarrow
H_{j,q'}\supseteq H_{J,q'}$ 
is trivially true, so we may further assume $j\notin J$. 
Now, by Lemma \ref{lm:period-poset} it suffices to show that 
the two conditions (a), (b) in Lemma \ref{lm:period-poset} 
hold true if and only if the same two conditions hold true with 
$q$ replaced by $q'$. 
%Again by the argument in the proof of 
%Theorem 3.1 of \cite{ktt}, 
%we can obtain the desired result because 
%$q,q'>q^*$. 
By part (i) of the present theorem with 
$J\cup \{ j\}$ regarded as the $J$ in (i), 
we know that (a) holds true if and only if (a) with $q$ replaced by 
$q'$ holds true.  
%because $q,q'>q^*\ge \tilde{e}(J\cup \{ j\})$. 
%
%Note that the elementary divisors of $(c_j,C_J)$ are 
%those of $C_{J'}$ for some nonempty $J'\subseteq [n]$  
%(specifically, $J'=J\cup \{j\}$); thus $\period_0={\rm lcm}\{ e(J): J
%\subseteq [n], \ J\ne \emptyset\}$ requires no change. 
%
%Note that the elementary divisors %$e_{J',1},\ldots,e_{J',\ell(J')}$ 
%of $(c_j,C_J)$ are 
%those of $C_{J'}$ with $J'=J\cup \{j\}$ and that they therefore 
%divide $\period_0$. 
By essentially the same discussion as above, 
we find 
%{\bf (check)} 
that (b) holds true if and only if (b) with $q$ replaced by 
$q'$ holds true, because $q,q'>q^*\ge e(J\cup\{ j\})$ 
for $J$ and $j$ such that 
$\rank C_{J\cup \{ j\}}=\rank C_{J}+1$. 
\qed 

\begin{rem} 
%Consider the central case $b=0$. 
%In that case, we have $q_T=0$ and $q^*=q_S$, and 
%Theorem \ref{th:periodicity-poset} reduces to Theorem 3.1 
%of \cite{ktt}. 
%Note that $H_{J,q}\ne \emptyset$ for any $q\in \bbZ_{>0}$ 
%and $J\subseteq [n]$ when $b=0$. 
Consider the central case $b=0$. 
In that case, we have $q_0=0$ and $q^*=\max\{q_1, %1, 
\max_{1\le j\le n}\gcd\{ c_{1j},\ldots,c_{mj} \}\}$. 
Note that $H_{J,q}\ne \emptyset$ for any $q\in \bbZ_{>0}$ 
and $J\subseteq [n]$ when $b=0$. 
\end{rem}

From Theorem \ref{th:periodicity-poset}, we obtain 
the following corollary.  

\begin{co}
Suppose $q, q'\in \bbZ_{>0}$ satisfy $q,q'>q^*$ 
and $\gcd\{\period_0,q\}=\gcd\{\period_0,q'\}$.  
Then $L_q$ is isomorphic to $L_{q'}$. 
In particular, the sequence of isomorphism classes of 
$L_q, \ q=1,2,\ldots$, is periodic in $q>q^*$ 
with a period 
$\period_0: \ L_q \simeq L_{q+\period_0}$ for $q>q^*$. 
\end{co}

%\

%Furthermore, 
We must emphasize that the periodicity of 
the sequence of isomorphism classes of 
$L_q, \ q \in \bbZ_{>0}$, in the sense of 
Definition \ref{def:iso} does not imply  
the periodicity of $|M(\calA_q)|, \ q\in \bbZ_{>0}$, with 
the same minimum period.  
In our example \eqref{eq:A=3400}, we can see from 
\eqref{eq:V<H<H} that 
\begin{equation*}
|M(\calA_q)|=
\begin{cases}
%|M(\calA_q)|=
q-1 
& \text{ if \ $\gcd\{ 12, q\}=1$,} \\ 
%|M(\calA_q)|=
q-2 
& \text{ if \ $\gcd\{ 12, q\}=2$,} \\ 
%|M(\calA_q)|=
q-3 
& \text{ if \ $\gcd\{ 12, q\}=3$,} \\ 
%|M(\calA_q)|=
q-4 
& \text{ if \ $\gcd\{ 12, q\}=4$,} \\ 
%|M(\calA_q)|=
q-4 
& \text{ if \ $\gcd\{ 12, q\}=6$,} \\ 
%|M(\calA_q)|=
q-6 
& \text{ if \ $\gcd\{ 12, q\}=12$.} \\ 
\end{cases}
\end{equation*}
So the minimum period of the quasi-polynomial 
$|M(\calA_q)|, \ q \in \bbZ_{>0}$, is $12$.  
On the other hand, the minimum period 
of the sequence of isomorphism classes of 
$L_q, \ q \in \bbZ_{>0}$, is $6$ by \eqref{eq:V<H<H}. 
This is due to the following fact: 
$L_q$ with $\gcd\{ 12, q\}=2$ 
is isomorphic to  
$L_q$ with $\gcd\{ 12, q\}=4$, 
while $|H_{2,q}|=2$ for $\gcd\{ 12, q\}=2$ is not equal to 
$|H_{2,q}|=4$ for $\gcd\{ 12, q\}=4$; 
similarly, 
$L_q$ with $\gcd\{ 12, q\}=6$ 
is isomorphic to  
$L_q$ with $\gcd\{ 12, q\}=12$, 
while $|H_{2,q}|=2$ for $\gcd\{ 12, q\}=6$ is not equal to 
$|H_{2,q}|=4$ for $\gcd\{ 12, q\}=12$.   

Concerning the coarseness of the 
intersection posets, we can obtain the following result: 
\begin{co}
Suppose $J_1\subseteq [n]$ and $J_2\subseteq [n]$ satisfy 
$H_{J_1,q}=H_{J_2,q}$ for some $q>q^*$. 
Then $H_{J_1,q'}=H_{J_2,q'}$ for any $q'>q^*$ such that 
$\gcd\{ \period_0, q' \}|\gcd\{ \period_0, q\}$. 
\end{co}
\Proof 
Just notice $d_{J,i}(q')|d_{J,i}(q), \ 1\le i\le \ell(J)$, 
for $J=J_1,J_2$ because of 
$\gcd\{ \period_0, q' \}|\gcd\{ \period_0, q\}$. 
Then the corollary follows from essentially the same argument 
as in the proof of Theorem \ref{th:periodicity-poset}. 
\qed 

\section{Example}
\label{sec:Bm}

In this section, we study the non-central 
arrangement $\hat{\calB}_m^{[0,a]} \ (a\in \bbZ_{>0})$ 
of Athanasiadis \cite{ath99} 
to illustrate our results. 

The arrangement $\hat{\calB}_m^{[0,a]} \ (a\in \bbZ_{>0})$ 
in $\bbR^m$ 
is a deformation of the Coxeter arrangement of type $B_m$, 
consisting of the hyperplanes defined by the following 
equations: 
\begin{eqnarray*}
x_i&=&0,1,\ldots,a, \quad 1\le i\le m, \\ 
x_i-x_j&=&0,1,\ldots,a, \quad 1\le i<j\le m, \\ 
x_i+x_j&=&0,1,\ldots,a, \quad 1\le i<j\le m.
\end{eqnarray*} 
We take the coefficient matrix $C$ and the vector $b$ so that 
$C$ consists of $1, -1$ or $0$.  

When $q$ is odd, Athanasiadis \cite{ath99} 
states in the proof of his Proposition 4.3 that 
the characteristic polynomial (%%HT 20080318
e.g., 
\cite{ort}) 
$\chi(\hat{\calB}_m^{[0,a]}, t)$ 
of the real arrangement $\hat{\calB}_m^{[0,a]}$ 
satisfies 
\begin{equation}
\label{eq:ch(B)Athanasiadis} 
\chi(\hat{\calB}_m^{[0,a]}, q)
= 
[y^{\frac{q-1}{2}-m}]
\left( \{ \phi_a(y)\}^{m+1} 
\sum_{j=0}^{\infty} 
(2j+1)^{m} y^{aj}
%\right) 
-
%[y^{\frac{q-1}{2}-m}]\left( 
f_{a-2}(y) \{ \phi_a(y)\}^{m-1} 
\sum_{j=0}^{\infty}
a_j'y^{aj}\right) 
\end{equation}
for all sufficiently large odd $q\in \bbZ_{>0}$, 
where $a'_0:=1$,  
\begin{equation}
\label{eq:a'j} 
a'_j:=\sum_{k=2}^m\binom{m}{k}(2^k-2)(2j-1)^{m-k}
=(2j+1)^m-2(2j)^m+(2j-1)^m, \quad j=1,2,\ldots,
\end{equation} 
\[
\phi_b(y):=1+y+y^2+\cdots +y^{b-1},  
\quad b\in \bbZ_{\ge 0}
\] 
(we understand that $\phi_0(y)=0$) 
and 
\[
f_{a-2}(y):=\sum_{s\ge 0, \ t \ge 0, \ s+2t \le a-2} y^{s+t}
=
\begin{cases}
%0 & \text{if $a=1$; } \\ 
\left\{ \phi_{\frac{a}{2}}(y) \right\}^2 & \text{if $a%\ge 2
$ is even; } \\  
\phi_{\frac{a-1}{2}}(y) \phi_{\frac{a+1}{2}}(y) & \text{if $a%\ge 3
$ is odd. }  
\end{cases}
\] 
(For a formal power series or a polynomial $F(y)$, %in $y$, 
we denote by 
$[y^k]F(y), \ k\in \bbZ_{\ge 0}$, the coefficient of $y^k$ in $F(y)$.) 
He obtained \eqref{eq:ch(B)Athanasiadis} 
by representing an $m$-tuple 
$z=(z_1, z_2, \ldots, z_m)\in \bbZ_q^m$ satisfying 
$z_i\ne0 \ (1\le i\le m)$ and 
$z_i \pm z_j \ne 0 \ (1\le i<j\le m)$ 
as a placement of $m$ integers 
$\epsilon_1 1, \epsilon_2 2,\ldots, \epsilon_m m \ 
(\epsilon_i=1 \text{ or} -1 \text{ for each } i=1,2,\ldots,m)$ 
and $(q-1)/2-m$ indistinguishable balls along a line, 
with an extra zero in the leftmost position. 

By inspecting his arguments, we can see that 
the right-hand side of \eqref{eq:ch(B)Athanasiadis} is 
equal to $|M((\hat{\calB}_m^{[0,a]})_q)|$ for all 
%(not necessarily large) odd $q\in \bbZ_{>0}$. 
odd $q\ge 2a+1$. 
Thus we have 
\begin{equation}
|M((\hat{\calB}_m^{[0,a]})_q)|
= 
[y^{\frac{q-1}{2}-m}]
\left( \{ \phi_a(y)\}^{m+1} 
\sum_{j=0}^{\infty} 
(2j+1)^{m} y^{aj}
-
f_{a-2}(y) \{ \phi_a(y)\}^{m-1} 
\sum_{j=0}^{\infty}
a_j'y^{aj}\right) 
\label{eq:M(B)q:odd}
%\nonumber 
%\\ 
%&=& [y^{\frac{q-1}{2}-m}]\left( \{ \phi_a(y)\}^{m+1} 
%\sum_{j=0}^{\left\lfloor \frac{\frac{q-1}{2}-m}{a} \right\rfloor} 
%(2j+1)^{m} y^{aj}\right) 
%\label{eq:M(B)q:odd} \\ 
%&& \qquad -[y^{\frac{q-1}{2}-m}]\left( f_{a-2}(y) \{ \phi_a(y)\}^{m-1} 
%\left(1+\sum_{j=1}^{\left\lfloor \frac{\frac{q-1}{2}-m}{a} \right\rfloor}a_j'y^{aj}\right) \right) \nonumber 
\end{equation}
for all odd 
%$q\in \bbZ_{>0}$, 
$q \ge 2a+1$. %, 
%where $\lfloor x \rfloor$ stands for the greatest integer not exceeding $x\in \bbR$.   
%From \eqref{eq:M(B)q:odd} we can see in particular that 
%$|M((\hat{\calB}_m^{[0,a]})_q)|=0$ for all odd $q\le 2m-1$,  
%\[
%|M((\hat{\calB}_m^{[0,a]})_q)|=
%\begin{cases} 
%1-0=1 & \text{if $a=1$; } \\  
%1-1=0 & \text{if $a\ge 2$ }  
%\end{cases}
%\]
%for $q=2m+1$,  
%and that 
%\[
%|M((\hat{\calB}_m^{[0,a]})_q)|=
%\begin{cases} 
%3^m-0=3^m & \text{if $a=1$; } \\ 
%(m+1)-\{(m-1)+0\}=2 & \text{if $a=2$; } \\ 
%(m+1)-\{(m-1)+1\}=1 & \text{if $a=3$; } \\  
%(m+1)-\{(m-1)+2\}=0 & \text{if $a\ge 4$ }  
%\end{cases}
%\] 
%for $q=2m+3$. 

Now, let us move on to the case of even $q$. 
By modifying the arguments of Athanasiadis \cite{ath99} 
for the case of odd $q$,  
we can count $|M((\hat{\calB}_m^{[0,a]})_q)|$ for even $q$ 
in the following way. 
Similarly to the case of odd $q$,  we consider a placement of 
$m$ integers $\epsilon_11,\ldots,\epsilon_mm$ and $q/2-m$ 
indistinguishable balls 
with an extra zero in the left most position. 
When the rightmost position %$q/2+1$ 
is occupied by an integer, its sign 
is always taken to be positive. 
Then we can obtain   
\begin{eqnarray}
\label{eq:M(B)q:even} 
|M((\hat{\calB}_m^{[0,a]})_q)|
&=& [y^{\frac{q}{2}-m}]\left( \{ \phi_a(y)\}^m 
\sum_{j=1}^{\infty}
%\left\{ \sum_{i=1}^{m}(2j)^{i-1}(2j-1)^{m-i} \right\}  
\left\{ (2j)^m-(2j-1)^m \right\} 
y^{aj}\right) 
\\ 
&& \quad + [y^{\frac{q}{2}-m-1}]\left( \{ \phi_a(y)\}^{m+1} 
\sum_{j=0}^{\infty} (2j+1)^{m} y^{aj}\right) 
\nonumber \\ 
&& \qquad -[y^{\frac{q}{2}-m-1}]\left( f_{a-3}(y) \{ \phi_a(y)\}^{m-1} 
\sum_{j=0}^{\infty}a_j'y^{aj}\right) 
\nonumber %\\ 
%&=& [y^{\frac{q}{2}-m}]\left( \{ \phi_a(y)\}^m 
%\sum_{j=1}^{\left\lfloor \frac{\frac{q}{2}-m}{a} \right\rfloor}
%\left\{ (2j)^m-(2j-1)^m\right\} 
%y^{aj}\right) 
%\nonumber 
%\\ 
%&& \quad + [y^{\frac{q}{2}-m-1}]\left( \{ \phi_a(y)\}^{m+1} 
%\sum_{j=0}^{\left\lfloor \frac{\frac{q}{2}-m-1}{a} \right\rfloor} (2j+1)^{m} y^{aj}\right) \nonumber \\ 
%&& \qquad -[y^{\frac{q}{2}-m-1}]\left( f_{a-3}(y) \{ \phi_a(y)\}^{m-1} 
%\left(1+\sum_{j=1}^{\left\lfloor \frac{\frac{q}{2}-m-1}{a} \right\rfloor}a_j'y^{aj}\right) \right) \nonumber 
\end{eqnarray}
for all even 
%$q\in \bbZ_{>0}$, 
$q \ge 2a+2$, 
where 
$a'_j \ (j=0,1,2,\ldots)$ are defined in \eqref{eq:a'j}  
%\[
%a'_j:=\sum_{k=2}^m\binom{m}{k}(2^k-2)(2j-1)^{m-k}=(2j+1)^m-2(2j)^m+(2j-1)^m, \ 
%j=1,2,\ldots,
%\]  
and 
\[
f_{a-3}(y)=\sum_{s\ge 0, \ t \ge 0, \ s+2t \le a-3} y^{s+t}
=
\begin{cases} 
%0 & \text{if $a=1,2$; } \\ 
\left\{ \phi_{\frac{a-1}{2}}(y) \right\}^2 & \text{if $a%\ge 3
$ is odd; } \\  
\phi_{\frac{a}{2}-1}(y) \phi_{\frac{a}{2}}(y) & \text{if $a%\ge 4
$ is even. }  
\end{cases}
\]
The first term on 
%of 
the right-hand side of \eqref{eq:M(B)q:even} 
corresponds to the first term %of 
on the right-hand side of \eqref{eq:M(B)q:odd} 
%in which 
with the restriction that 
the rightmost position %$q/2+1$ 
on the line 
is occupied by an 
integer 
%$i, \ 1\le i\le m$ 
(which is necessarily positive). 
%{\bf (Bunn-tou)In this term,} $\sum_{i=1}^{m}(2j)^{i-1}(2j-1)^{m-i}$ 
%represents the number of ways to 
%(i) choose $i\in \{ 1,2,\ldots,m\}$, 
%(ii) put each of 
%$1,\ldots, i-1$ in one of $j-1$ {\bf positions??? (between 
%$j$ $a$-blocks)} with an arbitrary sign or in one of 
%two {\bf positions??? (outside?? (of??) (the?) $j$ $a$-blocks)} 
%with a specified sign ($+$ in the position?? to the left of 
%the leftmost $a$-block, and $-$ in the position?? to the right 
%of the rightmost $a$-block), 
%and (iii) put each of $i+1,\ldots,m$ 
%in one of 
%$j-1$ {\bf positions??? (between 
%$j$ $a$-blocks)} with an arbitrary sign or in %one of 
%one {\bf position??? (to the left of?? the leftmost $a$-block)} 
%with a negative sign. 
%
%when there are $j$ $a$-blocks.  
The second term %of 
on the right-hand side of \eqref{eq:M(B)q:even} 
corresponds to the first term %of 
on the right-hand side of \eqref{eq:M(B)q:odd} 
%in which 
with the restriction that 
the rightmost position %$q/2+1$ 
is occupied by a ball.  
The third term %of 
on the right-hand side of \eqref{eq:M(B)q:even} 
corresponds 
%(repetition) 
to the second term %of 
on the right-hand side of \eqref{eq:M(B)q:odd}. 
%From \eqref{eq:M(B)q:even} we can see in particular that 
%$|M((\hat{\calB}_m^{[0,a]})_q)|=0$ for all even $q\le 2m$ and that 
%\[
%|M((\hat{\calB}_m^{[0,a]})_q)|=
%\begin{cases} 
%(2^m-1)+1-0=2^m & \text{if $a=1$; } \\ 
%0+1-0=1 & \text{if $a=2$; } \\  
%0+1-1=0 & \text{if $a\ge 3$ }  
%\end{cases}
%\] 
%for $q=2m+2$. 

As in Athanasiadis \cite{ath99}, let 
$S$ be the shift operator: 
\[
Sf(y):=f(y-1)
\] 
for polynomials $f(y)$. 

\begin{lm} 
\label{lm:[y^8q-l)]}
Let $p, a, m\in \bbZ_{>0}, \ l \in \bbZ_{\ge 0}$ and 
$h\in \bbZ_{\ge 0}$ with $h\le m-1$. 
Suppose $b\in \bbZ$ and 
$c\in %\bbZ_{ \times }:=
\bbZ\setminus \{ 0\}$ 
satisfy $c | ab$. 
Furthermore, assume $p\ge l+m(a-1)$.  
Then we have 
\begin{equation}
\label{eq:cj+b}
[y^{p-l}]\psi(y)\{\phi_a(y)\}^m\sum_{j=0}^{\infty}
(cj+b)^h y^{aj}
=
\frac{1}{a^{h+1}}\psi(S)\{ \phi_a(S)\}^m
S^{l-\frac{ab}{c}}(cp)^h 
\end{equation} 
for any polynomial $\psi(y)$. 
\end{lm}

\Proof 
It suffices to prove the lemma when 
$\psi(y)=1$. 
Define $c_{k;m,a}$ by 
$\{\phi_a(y)\}^m=\sum_{k=0}^{m(a-1)}
c_{k;m,a}y^k$. 
Then the left-hand side of \eqref{eq:cj+b} is  
\begin{equation*}
[y^{p-l}]
\left( 
\sum_{k=0}^{m(a-1)}\sum_{j=0}^{\infty}
c_{k;m,a}(cj+b)^h y^{aj+k}
\right) 
=
\frac{1}{a^h}
\sum_{k=p-l-aj, \ 0\le k\le m(a-1), \ 
j\in \bbZ_{\ge 0}}
c_{k;m,a}\{ c(p-l-k)+ab \}^h, 
\end{equation*}  
which can be written as  
\begin{equation}
\label{eq:1/a^h}
\frac{1}{a^h}
\sum_{k\equiv p-l \ (\mod a), \ 0\le k\le m(a-1)}
c_{k;m,a}\left\{ c\left(p-k-l+\frac{ab}{c}\right) \right\}^h
\end{equation}
because $p-l\ge m(a-1)$. 
Since $h\le m-1$, we have by Lemma 2.2 of 
Athanasiadis \cite{ath99} that 
\eqref{eq:1/a^h} is equal to 
\[
\frac{1}{a^h}\times \frac{1}{a}
\{ \phi_a(S)\}^m 
\left\{ c\left(p-l+\frac{ab}{c}\right) \right\}^h
=
\frac{1}{a^{h+1}}\{ \phi_a(S)\}^mS^{l-\frac{ab}{c}}
(cp)^h. 
\]
\qed 

If we use Lemma \ref{lm:[y^8q-l)]}, it is not hard 
to see that the constituent $P_2(q)$ 
of the characteristic quasi-polynomial 
$|M((\hat{\calB}_m^{[0,a]})_q)|$ for even $q$ is 
\begin{equation}
\label{eq:P2}
P_2(q)
=
\frac{\{\phi_a(S^2)\}^{m-1}S^{2m-a}(1-S^a)^2}{a^{m+1}}
\left( S^a\cdot \frac{1+S^a}{1-S^2}
+S^2 \cdot \frac{(1+S^a)^2}{(1-S^2)^2}
-f_{a-3}(S^2)S^2 \right)q^m.  
\end{equation}
Moreover, $|M((\hat{\calB}_m^{[0,a]})_q)|=P_2(q)$ for all 
even $q\ge 2a(m+1)$. 

\begin{rem}
%Actually, 
When $a$ is odd, we 
%%HT 20080318
cannot 
directly apply Lemma \ref{lm:[y^8q-l)]} to 
the terms on the right-hand side of \eqref{eq:M(B)q:even} 
because the condition $c|ab$ of the lemma is not met. 
But the result \eqref{eq:P2} itself obtained by formally applying the lemma is correct. 
Similarly, the exponent of each term of $(2j)^m-(2j-1)^m$ is $m$ and this fact violates the 
condition $h\le m-1$ of the lemma. 
But the degree of the polynomial $(2j)^m-(2j-1)^m$ in $j$ is $m-1$, and we can formally 
apply the lemma to each of $(2j)^m$ and $(2j-1)^m$ and then take the difference. 
The same remark applies to $a'_j=(2j+1)^m-2(2j)^m+(2j-1)^m, \ j\ge 1,$ as well. 
\end{rem}

When $a$ is odd, we can calculate  
\[
f_{a-3}(S^2)=\left( \frac{1-S^{a-1}}{1-S^2}\right)^2
\] 
and 
\[
P_2(q)=
\frac{\{\phi_a(S^2)\}^{m-1}S^{2m-a}(1-S^a)^2}{a^{m+1}}
\cdot 
\frac{S^a(1+S)^2}{(1-S^2)^2}q^m
%=
%\frac{\{\phi_a(S^2)\}^{m-1}S^{2m}}{a^{m+1}}
%\left( \frac{1-S^a}{1-S} \right)^2 q^m 
=
\frac{1}{a^{m+1}}S^{2m}\{ \phi_a(S^2)\}^{m-1}
\{ \phi_a(S)\}^2 q^m. 
\] 
By Propositions 4.2 and 4.3 of Athanasiadis \cite{ath99}, 
this is equal to $P_1(q)$ for odd $a$, 
so the period collapse 
occurs for odd $a$. 

When $a$ is even, on the other hand, 
we can obtain 
\[
f_{a-3}(S^2)=
\frac{(1-S^{a-2})(1-S^a)}{(1-S^2)^2}
\] 
and 
\[
P_2(q)=
\frac{2}{a^{m+1}}S^{2m}\{ \phi_a(S^2)\}^{m-1}
%\left( \frac{1-S^a}{1-S^2}\right)^2
\{ \phi_{\frac{a}{2}}(S^2) \}^2
(1+S^2)
q^m. 
\] 
By Propositions 4.2 and 4.3 of Athanasiadis \cite{ath99}, 
we know \[
P_1(q)=\frac{4}{a^{m+1}}S^{2m+1}
\{ \phi_{a}(S^2) \}^{m-1}
\{ \phi_{\frac{a}{2}}(S^2) \}^2
q^m
\] 
for even $a$. 
Hence, unlike the case of odd $a$, 
the period collapse 
does not occur for even $a$. 

Thus we have obtained the following theorem. 

\begin{theorem}
For %$\hat{\calB}_m^{[0,a]}$ with 
odd $a$, 
we have 
%the quasi-polynomial 
\begin{equation}
\label{eq:odd-a-M}
|M((\hat{\calB}_m^{[0,a]})_q)|=
\frac{1}{a^{m+1}}S^{2m}\{ \phi_a(S^2)\}^{m-1}
\{ \phi_a(S)\}^2 q^m
\end{equation} 
for all $q\ge 2a(m+1)-1$. 
For %$\hat{\calB}_m^{[0,a]}$ with 
even $a$, 
we have 
%the quasi-polynomial 
\begin{equation}
\label{eq:even-a-M}
|M((\hat{\calB}_m^{[0,a]})_q)|
=
\begin{cases}
\frac{4}{a^{m+1}}S^{2m+1}
\{ \phi_{a}(S^2) \}^{m-1}
\{ \phi_{\frac{a}{2}}(S^2) \}^2
q^m 
& \text{if $q$ is odd; } \\ 
\frac{2}{a^{m+1}}S^{2m}\{ \phi_a(S^2)\}^{m-1}
\{ \phi_{\frac{a}{2}}(S^2) \}^2
(1+S^2)
q^m 
& \text{if $q$ is even}  
\end{cases}
\end{equation}
for all $q\ge 2a(m+1)-1$. 
\end{theorem}

Using PARI/GP \cite{pg}, we have calculated $q_0$ for some %combinations of 
values of $m$ and $a$. 
Denote by $\bar{q}=\bar{q}(m, a)$ the greatest integer $q\in \bbZ_{>0}$ for which 
both sides do not agree in \eqref{eq:odd-a-M} or in \eqref{eq:even-a-M}. 
We list $q_0, 2a(m+1)-2$ and $\bar{q}$ 
for some combinations of values of $m$ and $a$:

\

$m=4, \ a=3$: 
\begin{center}
\begin{tabular}{c|c|ccccccccc}
$q_0$ & $2a(m+1)-2$ & $\bar{q}$ \\ \hline
30 & 28 & 21 
\end{tabular}
\end{center}

\

$m=4, \ a=4$: 
\begin{center}
\begin{tabular}{c|c|ccccccccc}
$q_0$ & $2a(m+1)-2$ & $\bar{q}$ \\ \hline
38 & 38 & 28 
\end{tabular}
\end{center}

\

$m=5, \ a=3$: 
\begin{center}
\begin{tabular}{c|c|ccccccccc}
$q_0$ & $2a(m+1)-2$ & $\bar{q}$ \\ \hline
42 & 34 & 27 
\end{tabular}
\end{center}

\

$m=5, \ a=4$: 
\begin{center}
\begin{tabular}{c|c|ccccccccc}
$q_0$ & $2a(m+1)-2$ & $\bar{q}$ \\ \hline
54 & 46 & 36 
\end{tabular}
\end{center}

%\

%$m=2$: 
%\begin{center}
%\begin{tabular}{l|cccccccccc}
%$a$ & 1 & 2 & 3 & 4 \\ \hline
%$q_0$ & 3 & 6 & 9 & 12
%\end{tabular}
%\end{center}

%\

%$m=3$: 
%\begin{center}
%\begin{tabular}{l|cccccccccc}
%$a$ & 1 & 2 & 3 & 4 \\ \hline
%$q_0$ & 6 & 10 & 18 & 22
%\end{tabular}
%\end{center}

%\

%$m=4$: 
%\begin{center}
%\begin{tabular}{l|cccccccccc}
%$a$ & 1 & 2 & 3 & 4 \\ \hline
%$q_0$ & 10 & 18 & 30 & 38
%\end{tabular}
%\end{center}

%\

%$m=5$: 
%\begin{center}
%\begin{tabular}{l|cccccccccc}
%$a$ & 1 & 2 & 3 & 4 \\ \hline
%$q_0$ & 14 & 26 & 42 & 54
%\end{tabular}
%\end{center}

\end{document}